\documentclass[11pt]{article}

\usepackage{authblk}

\usepackage{amsmath,amsthm,amssymb}
\usepackage{hyperref}
\usepackage{amsfonts}
\usepackage{url}

\usepackage[a4paper, margin=2.5cm]{geometry}

\newtheorem{theorem}{Theorem}[section]
\newtheorem{lemma}[theorem]{Lemma}
\newtheorem{corollary}[theorem]{Corollary}

\theoremstyle{definition}

\theoremstyle{remark}

\numberwithin{equation}{section}
\usepackage{enumitem}

\makeatletter
\newcommand*{\house}[1]{%
	\mathord{%
		\mathpalette\@house{#1}%
	}%
}
\newcommand*{\@house}[2]{%
	\dimen@=\fontdimen8 %
	\ifx#1\scriptscriptstyle\scriptscriptfont
	\else\ifx#1\scriptstyle\scriptfont
	\else\textfont\fi\fi
	3 %
	\sbox0{%
		$#1%
		\vrule width\dimen@\relax
		\overline{%
			\kern2\dimen@
			\begingroup 
			#2%
			\endgroup
			\kern2\dimen@
		}%
		\vrule width\dimen@\relax
		\mathsurround=1.5\dimen@ 
		$%
	}%
	\ht0=\dimexpr\ht0-\dimen@\relax
	\dp0=\dimexpr\dp0+2\dimen@\relax
	\vbox{%
		\kern\dimen@ 
		\copy0 %
	}%
}

\theoremstyle{plain}

\newtheorem{prop}[theorem]{Proposition}
\newtheorem*{claim}{Claim}
\newtheorem*{lemma*}{Lemma}
\theoremstyle{definition}

\newcommand{\ve}{\varepsilon}

\newcommand{\R}{\mathbb{R}}
\newcommand{\N}{\mathbb{N}}
\newcommand{\Z}{\mathbb{Z}}
\newcommand{\Q}{\mathbb{Q}}
\newcommand\Nr{\mathrm{N}}

\newcommand{\co}{\mathcal{O}}

\DeclareMathOperator\Tr{Tr}

\title{Even better sums of squares over quintic and cyclotomic fields}

\author{V\' \i t\v ezslav Kala\thanks{vitezslav.kala@matfyz.cuni.cz, ORCID: 0000-0001-5515-6801} }
\author{Pavlo Yatsyna\thanks{{p.yatsyna@matfyz.cuni.cz, ORCID: 0000-0003-2298-8446}}}
\affil{Charles University, Faculty of Mathematics and Physics, Department of Algebra, Sokolov\-sk\' a 83, 18600 Praha~8, Czech Republic
}
\date{}

\begin{document}

\maketitle

\begin{abstract}
We classify all totally real number fields of degree at most 5 that admit a universal quadratic form with rational integer coefficients; in fact, there are none over the previously unsolved cases of quartic and quintic fields. This fully settles the lifting problem for universal forms in degrees at most 5. The main tool behind the proof is a computationally intensive classification of fields in which every multiple of 2 is the sum of squares. We further extend these results to some real cyclotomic fields of large degrees and prove Kitaoka's conjecture for them.

\medskip 


\medskip 

\textbf{Funding.}
V.K. and P.Y. were supported by Czech Science Foundation, grant number 26-20514S. P.Y. was supported by Charles University, projects PRIMUS/24/SCI/010 and UNCE/24/SCI/022, Research Council of Finland (grant 351271, PI C. Hollanti).

\medskip 

\textbf{Acknowledgments.}
 	We are extremely grateful to James McKee for computing all the $9\,834\,226$ monic irreducible degree 5 polynomials with roots bounded in absolute value by $2+\sqrt{6}$. We thank Rudolf Scharlau for drawing our attention to Ralf Gerkmann’s Master’s thesis \cite{Ge} and for sharing it with us, and the anonymous referees for several helpful comments.

\medskip 

 	\textbf{Data availability.}
 All data is available upon request. Magma code available at \url{https://github.com/pav10/2O}.

\medskip 

\textbf{Statements and Declarations.}
 The authors have no competing interests to declare that are relevant to
the content of this article.
\end{abstract}

\section{Introduction}

Siegel's 1934 book \cite{Si book} on quadratic forms starts by noting that: \textit{``It is somewhat surprising that in a branch of
mathematics as old as the theory of quadratic
forms which originated among the ancient
Babylonians and has been intensively cultivated
during the last three
centuries by a succession
of mathematicians of the highest rank, including
Fermat, Gauss, Jacobi and Minkowski, any
fundamentally new ideas have been left to be
discovered.''}
Over the last 90 years, this surprise has not diminished, at the very least thanks to the celebrated 15- and 290-theorems of Conway--Schneeberger and Bhargava--Hanke \cite{Bh, BH}.

We will focus on the representability of algebraic integers as sums of squares. 
The sum of squares in a number field is always totally positive, and correspondingly, it is often the most interesting to restrict to the case when the field is totally real, as we do in this paper.
Although every totally positive algebraic number is the sum of four squares in a number field \cite[Satz~1]{Si1}, the corresponding question for rings of integers is much more difficult.

Maa{\ss} proved that the sum of three squares is \textit{universal over $\Q(\sqrt{5})$}, meaning that it represents all totally positive algebraic integers in $\Q(\sqrt{5})$ \cite{Ma}. Siegel later demonstrated 
that the sum of any number of squares in any totally real number field other than $\Q$ and $\Q(\sqrt{5})$ always misses some totally positive integers \cite[Theorem 1]{Si3}, and so is not universal. The study of other universal quadratic forms and their ranks has received considerable attention in recent years \cite{BK, CO, HHX, KT, KTZ, KYZ, Man, Ya}.

On the other hand, for any totally real number field $K$, all totally positive integers divisible by some fixed positive integer $m$ (that depends only on $K$) can be represented as the sum of integral squares \cite[Satz 9]{Si1}. However, for given positive integers $m$ and $d$, there are only finitely many totally real number fields of degree $d$ for which all totally positive integers divisible by $m$ can be represented as the sum of squares \cite[Theorem 2]{KY3}. In this article, we elaborate on the case $m = 2$ and classify all number fields of degree $d$ up to five for which all even (i.e., those divisible by $2$) totally positive integers can be represented as the sum of squares.

\begin{theorem}\label{Thm:o2}
Let $K$ be a totally real number field of degree $d\le 5$. 
Every element of $2\co_K^+$ is the sum of squares of integers if and only if $K$ is isomorphic to \begin{equation*}\label{2fields}
\Q,\ \Q(\sqrt{2}),\ \Q(\sqrt{3}),\ \Q(\sqrt{5}),\ \Q(\zeta_7+\zeta_7^{-1}),\  \Q(\rho),\ \Q(\sqrt{2},\sqrt{5}),\text{ or }\ \Q\left(\sqrt{(5+\sqrt{5})/2}\right),
\end{equation*}
where $\rho$ is a root of $x^3-4x-2.$
\end{theorem}

We first prove an explicit version of \cite[Theorem 7]{KY3} by establishing (in Theorem \ref{thm:hat}) 
that the house of the generator of fields with $m=2$ is bounded. 
This is then followed by a significant computational part of the proof that involves dealing with tens of millions of number fields.

Theorem \ref{Thm:o2} has striking implications for the \textit{lifting problem for universal forms}, introduced 
by Kala--Yatsyna
\cite{KY1} (see also \cite{XZ}):
While we have already seen Siegel's result that the sum of squares is very rarely universal, the lifting problem asks the analogous question about the universality of other \textit{$\Z$-forms}, i.e., positive definite quadratic forms whose coefficients are rational integers. Some examples of universal $\Z$-forms over $\Q(\sqrt 5)$ were found by Chan--Kim--Raghavan \cite{CKR} and Deutsch \cite{De}, before Kala--Yatsyna proved \cite[Theorem 1.1]{KY1} that $\Q(\sqrt 5)$ is the only real quadratic field that admits a universal $\Z$-form, and obtained some partial results in other small degrees  \cite[Theorem 1.2]{KY1}. Gil Mu\~ noz--Tinkov\' a covered certain cubic fields \cite{GMT}, until all these results were significantly expanded by Kim--Lee \cite{KL}, who completely characterised cubic and biquadratic fields with a universal $\Z$-form.

In the present paper, we completely solve the lifting problem in degrees $d\leq 5$, thus fundamentally strengthening all the previous theorems.

\begin{theorem}\label{cor:z-forms}
	Let $K$ be a totally real number field of degree $d\le 5$. 
	There is a $\Z$-form that is universal over $K$ if and only if $K$ is isomorphic to $\Q,\ \Q(\sqrt{5})$, or $\ \Q(\zeta_7+\zeta_7^{-1}).$
\end{theorem}

The implication from right to left is known, as some examples of universal $\Z$-forms over the three fields listed above are $x^2+y^2+z^2+w^2$ (Lagrange), $x^2+y^2+z^2$ \cite{Ma}, and $x^2+y^2+z^2+w^2+xy+xz+xw$ \cite[Theorem~1.2]{KY1}, respectively.
The other implication will be proved in Section~\ref{se:4}.

Theorem~\ref{Thm:o2} also implies a new result on number fields satisfying the local-global principle for representations of integers as sums of integral squares. Scharlau \cite{Sc1, Sc2} classified all such totally real number fields in degrees up to 4, and our Corollary~\ref{th:exc} shows that there are no further such fields in degree 5.

By studying the basis of quartic number fields, we also prove a quantitative result (Corollary~\ref{cor:quart}) on the finiteness of nonprimitive quartic number fields such that all elements of $m\co_K^+$ are sums of squares.
This is a much more explicit version of \cite[Theorem~2]{KY3} (that concerned a finiteness result for all fields of a given degree). 
Our proof (in Corollary~\ref{cor:quart}) is different and utilizes a unique structure of quartic number fields that contain real quadratic subfields. 

The proof of Theorem~\ref{Thm:o2} is based on a fairly large number of computations concerning elements with a small house (see the beginning of the next section for the definition). While it may be possible to similarly deal with fields of slightly larger degrees, it is unlikely that this method will lead to a result valid for all degrees $d$. 
This prompted us to investigate totally real cyclotomic number fields, as they contain a proper integer with a house bounded by 2, and their degree can be arbitrarily large. We demonstrate the following.

\begin{theorem}\label{thm:C2} Let $q\geq 5$ be an odd prime or a power of two.
	Let $K_q=\Q(\zeta_q+\zeta_q^{-1})$ be the maximal totally real subfield of the $q$-th cyclotomic field. 
	
	If every element of $2\co_K^+$ is the sum of squares of integers in $K_q,$ then $q=5,\ 7,$ or $8$.
\end{theorem} 
 
 This is proved in Propositions~\ref{cor:2squares} and \ref{cor:psquares}. 
 For other small composite values of $q$, we can also use Theorem~\ref{Thm:o2}. In particular, we see that all elements of $2\co_{K_q}^+$ are sums of squares also for $q=12,\ 20$ (e.g., $\Q(\sqrt{3})=\Q(\zeta_{12}+\zeta^{-1}_{12})$ and $\Q\left(\sqrt{(5+\sqrt{5})/2}\right)=\Q(\zeta_{20}+\zeta^{-1}_{20})$). 

 Thus, the first values for which we technically do not know whether all elements of $2\co_{K_q}^+$ are sums of squares are $q=21$, $q=27$, and $q=28$. (Although, practically, it is easy to check computationally all such low-degree cases.)  
 
We conclude the paper by linking the above result to Kitaoka's conjecture, the assertion that there can only exist finitely many totally real number fields with universal ternary quadratic forms. The conjecture is known to hold for all odd degree number fields \cite{EK}, and for number fields of bounded degree \cite[Theorem~3]{KY3}. We confirm here Kitaoka's conjecture for an infinite family of maximal totally real subfields of cyclotomic fields.

\begin{theorem}\label{thm:CK} Let $q\geq 5$ be an odd prime or a power of two.
	Let $K_q=\Q(\zeta_q+\zeta_q^{-1})$ be the maximal totally real subfield of the $q$-th cyclotomic field.

A ternary universal quadratic form exists over $K_q$ if and only if 
\begin{itemize}
	\item $q=5$ (then $K_q=\Q(\sqrt 5)$), or
	\item $q=8$ (then $K_q=\Q(\sqrt 2)$).
\end{itemize}
\end{theorem} 
 This is proved in Theorems~\ref{thm:2-K} and \ref{thm:p-K} by studying the representation of specific indecomposable integers. Theorems~\ref{thm:C2} and \ref{thm:CK} (and their proofs) are largely influenced by parts 5A and 5B of Scharlau's thesis \cite{Sc2}. 
 
In Section~\ref{s:2}, we introduce most of the necessary notation and conventions. 
Section~\ref{s:3} deals with the case of number fields with even integers represented by the sum of squares, Section~\ref{se:4} then proves Theorem~\ref{cor:z-forms},
and Section~\ref{s:4} expands on the quartic case. The second half of the paper concerns the real cyclotomic fields:
we deal with the powers of two in Section~\ref{s:6}, and with odd primes in Section~\ref{s:7}.

Let us conclude the Introduction with a tentative conjecture. As described at the end of Section~\ref{s:3}, we also searched for totally real number fields $K$ of degree $\leq 18$ in which all elements $2\co_K^+$ are sums of squares, but we did not find any in addition to the fields of degrees $\leq 4$ listed in Theorem~\ref{Thm:o2}. This suggests that this list may be complete for \textit{all} degrees, in which case there would also be no further fields in the situation of Corollary~\ref{th:exc}. The affirmative answer to the above conjectures would not directly lead to a similar conclusion to the lifting problem (that is, there are no further fields with universal $\Z$-forms besides those from Theorem \ref{cor:z-forms}) as there exist $\Z$-forms that are not represented by the sums of squares of integral linear forms in ranks greater than 5 (see, for example, \cite{Mo}). 

Finally, since this paper was finished and posted on arxiv in 2024, it has inspired several follow-up works that build on and expand our results: \cite{KK, KKK, KKL}.
 	
\section{Generalities}\label{s:2}
Let $K$ be a totally real number field of degree $d$ over $\Q$ with the ring of integers $\mathcal{O}_K$. Let $\sigma_1,\dots,\sigma_d:K\hookrightarrow\R$ be the distinct real embeddings of $K$ and let $\sigma=(\sigma_1,\dots,\sigma_d)^t:K\hookrightarrow\R^d$ be the corresponding embedding of $K$ into the Minkowski space. The \textit{norm} and \textit{trace} of $\alpha \in K$ are defined by $\Nr_{K/\Q}(\alpha)=\sigma_1(\alpha)\cdots\sigma_d(\alpha)$ and $\Tr_{K/\Q}(\alpha)=\sigma_1(\alpha)+\dots+\sigma_d(\alpha)$, respectively. Additionally, let $\house{\alpha}= \max_i(|\sigma_i(\alpha)|)$ be the \textit{house} of $\alpha$. As $\Nr_{K/\Q}(\alpha)\in \Z$ for all $\alpha \in \mathcal{O}_K$, if we also assume that $\alpha$ is nonzero, then it follows that $\house{\alpha}\ge 1$. 

An element $\alpha\in K$ is \textit{totally positive} if $\sigma_i(\alpha)>0$ for all $i$. The set of all totally positive elements of $K$ is denoted $K^+$; if $E$ is a subset of $K$, then $E^+=E\cap K^+$. We say that $\alpha$ is \textit{totally greater than} $\beta$ (denoted $\alpha\succ\beta$) if $\alpha-\beta\in K^+$, for $\alpha,\beta\in K$.  A totally positive integer $\alpha\in\co_K^+$ is \textit{indecomposable} if there does not exist $\beta\in\co_K^+$ such that $\alpha\succ\beta$.   An element $\alpha\in K$ is \textit{proper} if it does not lie in a proper subfield of $K$, that is, if $\Q(\alpha)=K$.

The maximal totally real subfield of the $q$-th cyclotomic field we denote by $K_q=\mathbb{Q}(\zeta_q+\zeta_q^{-1})$, where $\zeta_q=e^{2\pi i/q}$. Its degree is  $d=[K_q:\mathbb{Q}]=\frac{\varphi(q)}{2}$. When $q$ is $2^n$ (for $n\ge 3$) or an odd prime number $p$, the degree equals $d=2^{n-2}$ or $\frac{p-1}{2}$, respectively. We write $\mathcal{O}_q$ for the ring of integers of $K_q$.

We set $\omega=\zeta_q+\zeta_q^{-1}$ and $\omega_j=\zeta_q^j+\zeta_q^{-j}$ for $0\neq j\in\mathbb{Z}$, with $\omega_0=1$. It can be seen that $-1=\omega_1+\dots+\omega_{d}$ and $\house{ \omega_j }<2$. It is well known that $\mathcal{O}_q=\mathbb{Z}[\omega]$ (see, for example, \cite[Proposition 2.16]{Wa}). Moreover, $(\omega_0,\omega_1,\dots,\omega_{d-1})$ forms a $\mathbb{Z}$-basis of $\mathcal{O}_q$, for by induction on $j$, it can be easily shown that $\omega^j$ lies in the $\mathbb{Z}$-span of $1,\omega,\omega^2,\ldots,\omega^{j-1},\omega_j$.

For quadratic forms, we follow the notation and terminology from \cite{O1}. Specifically, let $V=K^n$ be $n$-ary \textit{quadratic space} over $K$ with its \textit{bilinear form} $B:V\times V\longrightarrow K$ and the corresponding \textit{quadratic form} $Q$, that is, $B(v,v)=Q(v)$. We use $\langle a_1,\ldots,a_k \rangle$ to denote the \textit{diagonal} quadratic form $a_1 x_1^2+\cdots+a_k x_k^2$ where $a_i\in K.$ Further, we denote the \textit{orthogonal sum} of two quadratic forms $Q_1$ and $Q_2$ by $Q_1\perp Q_2$. A quadratic form $Q$ can be associated with the \textit{Gram matrix} $M=(B(e_i,e_j)),$ where $e_i$s are the standard basis. We use $Q\cong H$ to denote that quadratic forms $Q$ and $H$ are \textit{integrally equivalent}, i.e., there exists an invertible integer matrix $X$ such that $M=XNX^t,$ where $M$ and $N$ are Gram matrices of $Q$ and $H$, respectively. We may lightly abuse the notation and write $Q\cong M$ to indicate that the quadratic form $Q$ is associated with the Gram matrix $M$.

The quadratic form is \textit{totally positive definite} if $Q(v)\succ 0$ for all nonzero $v \in V$. We say that $Q$ \textit{represents} an algebraic integer $\alpha$ (over $\co_K$) if there exists $v \in \co_K^n$ such that $Q(v)=\alpha$. A quadratic form is \textit{integral} if $Q(v)\in\co_K$ for every $v \in \co_K^n$, and \textit{classical} if $B(v,w)\in\co_K$ for every $v,w \in \co_K^n$. A vector $v\in \co^n_K$ is \textit{primitive} if $v\not= \beta w$ for every $w\in \co^n_K$ and nonunit $\beta\in \co_K$. We say that a totally positive definite quadratic form is \textit{universal} (over $\co_K$ or $K$) if it represents all elements of $\co_K^+$. 

A useful and well-known result is the ``splitting-off units'' (see, e.g., \cite[82:15]{O1} or \cite[Proposition 3.4]{Ka}  for a self-contained short proof).

\begin{lemma}\label{lemma:unit}
	Let $K$ be a totally real number field and $Q$ be a classical quadratic form over the ring of integers of $K$ that represents a unit $u$. Then
	\[Q\cong \langle u \rangle \perp Q',\] where $Q'$ is a classical quadratic form of rank smaller than $Q$.
\end{lemma}

Later we shall use the following results that show that elements of small trace and norm are quite special.

\begin{theorem}\label{thm:1}\label{cor:3/2}\label{cor:5/2}
	
	 Let $K$ be a totally real number field of degree $d$ and take $\gamma\in\co_K^+$.
	
a) If $\Tr_{K/\Q}(\gamma) \le\frac{3}{2}d$, then $\gamma=1$ or $\phi+1$, where $\phi^2-\phi-1=0$.

b) If $\Tr_{K/\Q}(\gamma) <\frac{3}{2}d$, then $\gamma=1$.

c) If $\Tr_{K/\Q}(\gamma) <\frac{5}{2}d$, then $\gamma$ is indecomposable or $\gamma \in \Z$.
\end{theorem}

\begin{proof}
 Parts a) and b)  are {\cite[Theorem III]{Si2}}.
	
c)	Assume that $\gamma\not\in\Z$ and that $\gamma$ decomposes as $\gamma=\alpha+\beta$, $\alpha,\beta\in\co_K^+$. As $\Tr_{K/\Q}(\alpha)\geq d$, we have $\Tr_{K/\Q}(\beta)< \frac{3}{2}d$, and so $\beta=1$ by the previous corollary. Symmetrically we have $\alpha=1$, which is in contradiction with our assumption that $\gamma\not\in\Z$.	
\end{proof}

\begin{prop}\label{lemma:norm}  Let $K$ be a totally real number field of degree $d$.
	Let $c>1$ be such that for every $\alpha \in\co_K^+$ we have: if $\Tr_{K/\Q}(\alpha) <cd$, then $\alpha=1$. Let $\gamma\in\co_K^+$ be a primitive element (that is, $n\nmid\gamma$ for all $n\in\Z_{>0}$) such that $\Nr_{K/\Q}(\gamma) <2^{d-1}(1+c)$. Then $\gamma$ is indecomposable.
\end{prop}	

Note that by part b) of Theorem \ref{cor:3/2} we can always take $c=3/2$.

\begin{proof}
	Assume for contradiction that $\gamma=\alpha+\beta$, $\alpha\neq\beta$. Then
	\[\Nr_{K/\Q}(\alpha+\beta)=\prod_i(\sigma_i(\alpha)+\sigma_i(\beta)).\]
	The product distributes as the sum of $2^d$ terms, each of which is of the form $X_1\cdots X_d$, where $X_i=\sigma_i(\alpha)$ or $\sigma_i(\beta)$. By the symmetry of this sum, we can write it as$$\frac 1d \sum \Tr_{K/\Q} (X_1\cdots X_d),$$
where we are summing over the same $2^d$ terms as before.
	
	Using the trace estimate from the hypothesis of the proposition, we want to show that many of the products $X_1\cdots X_d=1$.
	However, if we get $X_1\cdots X_d=Y_1\cdots Y_d=1$ with $X_i=Y_i$ for all $i\neq j$ and $X_j=\sigma_j(\alpha)$, $Y_j=\sigma_j(\beta)$, this would imply that $\sigma_j(\alpha)=\sigma_j(\beta),$ and so $\alpha=\beta$, a contradiction.
	
	So now this becomes a combinatorial exercise of suitably pairing up the products $X_1\cdots X_d$.
	
	As a first attempt, we can consider $\sigma_1(\alpha)X_2\cdots X_d$ and $\sigma_1(\beta)X_2\cdots X_d$: there are $2^{d-1}$ such pairs, and for each pair, at most one element equals 1.
	
	Thus, for each pair, at most one element has trace $d$, and so 

 \[\Tr_{K/\Q}(\sigma_1(\alpha)X_2\cdots X_d)+\Tr_{K/\Q}(\sigma_1(\beta)X_2\cdots X_d)\geq (1+c)d.\]
	
 In total, we get 
\[\Nr_{K/\Q}(\alpha+\beta)=\prod_i(\sigma_i(\alpha)+\sigma_i(\beta))=\frac 1d \sum \Tr_{K/\Q} (X_1\cdots X_d)\geq 2^{d-1}(1+c). \qedhere\]
\end{proof}
Note that Proposition \ref{lemma:norm} is better than \cite[Lemma~2.1b]{KY1} as soon as $c>1$. Further, if just one product (without loss of generality) \[\sigma_1(\alpha)\cdots\sigma_i(\alpha)\sigma_{i+1}(\beta)\cdots\sigma_d(\beta)=1\] with $1\leq i<d$, then for $\kappa:=\alpha/\beta$ we have
\[\sigma_1(\kappa)\cdots\sigma_i(\kappa)=\frac 1{\Nr_{K/\Q}(\beta)}\in\Q,\]
and so $\kappa$ lies in a proper subfield of $K$. In particular, if $K$ is primitive (has no proper subfields), then elements of the norm $\Nr_{K/\Q}(\gamma) <(2^{d}-2)c+2$ must be indecomposable.

Finally, the above may be improved by replacing the trace condition with the minimum length of proper totally positive units in the number field, where the length of a totally positive integer $\alpha$ is $\Nr_{\Q(\alpha)/\Q}(\alpha+1)$. (For example, for almost all totally positive integers of degree $d$, the length is greater than $2.364950^d$ \cite[Theorem~1]{MW}. There are also bounds for length in terms of the trace, see \cite[Proposition~1]{FSE}.)

Finally, the following norm bound will be useful.

\begin{lemma}[{\cite[3.1]{O2}}]\label{lemma:norm_bound} Let $K$ be a totally real number field of degree $d$.
For all $\alpha_1,\alpha_2\in \co_K^+$ we have \[\Nr(\alpha_1+\alpha_2)^{1/d}\ge \Nr(\alpha_1)^{1/d}+\Nr(\alpha_2)^{1/d},\]
with equality if and only if $\alpha_1/\alpha_2\in\Q$.
\end{lemma}

\section{Multiples of 2 represented as sums of squares}\label{s:3}

We start with a lemma that states that for odd-degree number fields, if all even totally positive integers are sums of squares, then all totally positive units are squares.
\begin{lemma} 
Let $K$ be a totally real number field of odd degree. If there exists a nonsquare unit $ \varepsilon \in \co_K^+$, then $2 \varepsilon$ is not a sum of squares. 
 \end{lemma}
 \begin{proof}
Using Lemma \ref{lemma:norm_bound}, it is obvious that the only possibilities for $2 \varepsilon$ to be represented as the sum of squares are $2 \varepsilon=\beta_1^2+\beta_2^2$ or $\beta^2$. This follows from the fact that $2=\Nr_{K/\Q}(2 \varepsilon)^{1/d}\ge \sum^k \Nr_{K/\Q}(\beta_i)^{2/d}\ge k$, for nonzero $\beta_i \in \co_K$, with equality on the left if and only if all $\beta^2_i$ are equal. If $2 \varepsilon=\beta_1^2+\beta_2^2$, then by the above inequality we have that $\beta^2_1=\beta^2_2= \varepsilon.$ Otherwise, we have that there exists an element of norm $2^{d/2}$, which is impossible for odd-degree number fields.
 \end{proof}
    
	Let us examine representing elements of $2\co_K^+$ as the sums of squares. Let $X>0$ and define $E=E_X$ as the subfield of $K$ that is generated (over $\Q$) by elements $\alpha\in\co_K$ with $\house{\alpha}<X$. 
	
	If $E\neq K$, let $Y$ be the smallest real number such that there is $\beta_0\in\co_K\setminus E$ with $\house{\beta_0}=Y$. By adding a suitable  rational integer $n$ to $\beta_0$, we obtain $\beta=\beta_0+n\in\co_K^+\setminus E$ 
	with $\house{\beta}<2Y+1$.
	
	Assume that $2\beta=\sum x_i^2$. We have $x_i\not\in E$ for some $i$ (as $\beta\not\in E$), and so
	$$4Y+2>\house{2\beta}\geq \house{x_i^2}\geq Y^2.$$
	This is possible only if $Y\leq 2+\sqrt 6$. With the given bound we can adjust the choice for $n$ in construction of $\beta$. Specifically, we can consider $\beta_0+5$, which yields the inequality
    \[2Y+10\ge Y^2,\]
    that is satisfied for $Y\le 1+\sqrt{11}$, and so also $X\leq 1+\sqrt{11}$. However, we can choose $X$ arbitrarily, and so we get a contradiction if $E_{1+\sqrt{11}}\neq K$. Thus, we have proved the following.
	
	\begin{theorem}\label{thm:hat}
		If every element of $2\co_K^+$ is the sum of squares, then $K$ is generated (as a field extension of $\Q$) by elements $\alpha\in\co_K$ with $\house{\alpha}<1+\sqrt{11}$.
	\end{theorem}
	
	For every $d\in\N$ and $B>0$, there are only finitely many algebraic integers $\alpha$ of degree $d$ such that $\house{\alpha}<B$. Therefore, there are finitely many $K$ satisfying the above theorem (in each degree), and there are sufficiently few of them that the candidate fields $K$ could be checked computationally. We carried out these calculations in Magma and PARI/GP for degrees $d\leq 5$.
	 Note that most of the number fields listed in Theorem \ref{Thm:o2} are isomorphic to the maximal real subfields of cyclotomics, that is, $\Q(\sqrt{2})\cong \Q(\zeta_{8}+\zeta_{8}^{-1}),\ \Q(\sqrt{3})\cong \Q(\zeta_{12}+\zeta_{12}^{-1}),\ \Q(\sqrt{5})\cong \Q(\zeta_{5}+\zeta_{5}^{-1}),\ \Q(\zeta_7+\zeta_7^{-1}),$ and $ \Q\left(\sqrt{(5+\sqrt{5})/2}\right)\cong \Q(\zeta_{20}+\zeta_{20}^{-1})$. This is perhaps not too surprising given the small house bound of the generator element.
	\begin{proof}[Proof of Theorem $\ref{Thm:o2}$]\
		
		\textbf{Implication ``$\Leftarrow$''.} This is straightforward to check using an upper bound \cite[Theorem 5]{KY3} on the norm of indecomposables. 
For if $2\alpha$ is the sum of squares for each indecomposable $\alpha$, then each element of $2\co_K^+$ is the sum of squares, as each totally positive integer is the sum of indecomposables. Since we know that each indecomposable $\alpha$ satisfies $\Nr_{K/\Q}(\alpha)\leq\Delta_K$ \cite[Theorem 5]{KY3} (where $\Delta_K$ is the discriminant of $K$), there are only finitely many classes of indecomposables up to multiplication by squares of units (which are easily found), so it suffices to check whether a representative of each of these classes is the sum of squares.
		
		However, this implication is known in all cases other than the number field generated by the root $\rho$ of $x^3-4x-2$: Scharlau \cite[Satz]{Sc1} proved that all our fields except $\Q(\zeta_7+\zeta_7^{-1})$ and $\Q(\rho)$ satisfy the following: \textit{If $\alpha\in\co_K^+$ is congruent with a square modulo $2\co_K$, then $\alpha$ is the sum of squares in $\co_K$.} As every element of $2\co_K^+$ is indeed a square modulo $2\co_K$, all these elements are sums of squares.      	 
		Furthermore, in \cite[proof of Lemma 6.1]{KY1} it was shown that twice any indecomposable element in $\Q(\zeta_7+\zeta_7^{-1})$ is the sum of squares. As each totally positive element is the sum of indecomposables, this implies the claim for $\Q(\zeta_7+\zeta_7^{-1})$.
		
		It remains to consider $\Q(\rho)$. Its discriminant equals $148$, and all its totally positive units are squares. Thus, using the norm bound mentioned above \cite[Theorem 5]{KY3}, we check that twice each indecomposable element is represented by the sum of squares. Specifically, we find all the integers of norms between $2$ and $148$ and directly confirm that for each totally positive integer $\alpha$, $2\alpha$ is represented by the sum of squares.	
		
		\medskip
			
		\textbf{Implication ``$\Rightarrow$''.} By Theorem \ref{thm:hat}, it suffices to show that in all number fields $K$ generated by the elements with the house bounded by $1+\sqrt{11}$, the condition that every element of $2\co_K^+$ is the sum of squares is not satisfied (unless it is one of the number fields listed in Theorem \ref{Thm:o2}). 
		
		The quadratic case was completely classified in \cite[Theorem~1.3]{KY2}. So it is left to check the cubic, quartic, and quintic number fields. In view of Theorem \ref{thm:hat}, we begin by finding all the algebraic integers $\alpha$ of degrees three, four, and five such that $\house{\alpha}<1+\sqrt{11}$. For this, we apply Robinson's algorithm (as described in Sections 3.1 and 3.2 of \cite{MS}) by computing all irreducible monic integer polynomials with only real roots that are bounded in absolute value by $1+\sqrt{11}$. Actually, we ran our computations up to the slightly larger bound $2+\sqrt 6$ and found $2\,885$ cubic, $190\,084$ quartic, and $983\,422$ quintic polynomials.
		
		For each such polynomial, we construct the corresponding number field.
		We test whether twice an element of the integral basis (when totally positive) is represented as the sum of squares (six, seven and eight squares, respectively, following from the bounds on the Pythagoras number for cubic, quartic and quintic number fields in \cite[Corollary~3.3]{KY1}). For the number fields that pass this test, we use the bounds of the norm of Theorem 5 in \cite{KY3} to find counterexamples. These simple procedures sieve out all but four polynomials, corresponding to the number fields from the statement of our theorem.
		
		However, this is not yet the end of the proof. In this way, we only found the fields $\Q(\alpha)$ where $\house{\alpha}<1+\sqrt {11}$. However, Theorem \ref{thm:hat} allows the field $K$ to be generated by \textit{several} elements of a small house (while every single generator of the field would be large), i.e., $K=\Q(\alpha_1,\alpha_2, \dots)$. This is not too much of an issue since we are dealing with degrees $d\leq 5$ since each $\Q(\alpha_i)$ must be a proper subfield of $K$, which is impossible when $d=1,2,3,5$. When $d=4$, the only cases that we need to consider are the biquadratic cases $\Q(\alpha_1,\alpha_2)$, where $\alpha_i$ are quadratic irrationals with $\house{\alpha_i}<1+\sqrt{11}$ for $i=1,2$. They generate 23 distinct quadratic fields, for which we need to consider their biquadratic composite fields.
Again, considering each element of the integral basis, we exclude all but 11 of them. One of these fields is $\Q(\sqrt{2},\sqrt{5})$; and for each of the remaining 10 fields, we finally find an element $\alpha$ such that $2\alpha$ is not the sum of squares.
	\end{proof}
 Let us remark that we had to pay particular attention to the field $\Q\left(\sqrt{(7+\sqrt{5})/2}\right)$ for which almost all elements of the ``small'' norm in $2\co_K^+$ are sums of squares.

We were able to carry out the calculations for degrees less than five in Magma \cite{Mag} (on a PC) and complete them in a few hours. The code is available on GitHub \url{https://github.com/pav10/2O}. In this case, the house bound is quite effective, and there are only a few fields that need to be checked. The exponential complexity of the computations makes it a demanding task to even find all viable polynomials without an effective algorithm for higher degrees. We are extremely grateful to James McKee for computing all monic irreducible degree five polynomials with roots bounded in absolute value by $2+\sqrt{6}$. There are $9\,834\,226$ such polynomials, with traces ranging from $-13$ to $13$.

It should be noted that an element with a house significantly less than the bound provided in Theorem \ref{thm:hat} generates each of the number fields described in Theorem \ref{Thm:o2}. In light of this, we looked further into the number fields of degrees up to $18$ from \cite[Appendix~A]{CDV} that correspond to small-span polynomials (that is, polynomials for which the difference between the greatest and smallest roots is less than $4$, excluding those whose roots are strictly between $0$ and $4$. See also \cite{Ro}). No additional (higher-degree) examples of number fields where $2\co_K^+$ is the sum of squares were found. We did not outright check the corresponding number fields for unrepresented integers, as the bounds for Pythagoras number of number fields of degrees $\ge 7$ are enormous (see, e.g., \cite{KO1, KO2} and Corollary~3.3 in \cite{KY1}). Instead, we found integers whose double is not represented by $d+3$ squares, where $d$ is the degree of the associated number field. After, we confirmed with norm bounds (from Lemma \ref{lemma:norm_bound}), that the number of squares considered was sufficient. Retrospectively, it was overkill, as for all examples, at most four squares would suffice. This perhaps provides weak evidence for the possible conjecture that the list in Theorem \ref{Thm:o2} is complete (for any degree $d$).

\section{$\Z$-forms and exceptional elements}\label{se:4}

Let us now discuss the two applications of Theorem \ref{Thm:o2} mentioned in the Introduction.

First, however, we obtain the following finiteness result for the existence of universal $\Z$-forms that is
 an improvement of  \cite[Corollary 9]{KY3}. 
\begin{theorem}
	For each positive integer $d$ there are only finitely many totally real number fields of degree $d$ with a universal $\Z$-form over $K$.
\end{theorem}
\begin{proof} Let $K$ be a totally real field of degree $d$ and let $Q$ be a universal $\Z$-form over $K$. By \cite[Proposition~3.1]{KY1}, we know that each totally positive integer in $\co_K$ is represented by a rational quadratic form of rank at most $d$. There exists an absolute bound, say $r$, for the minimal rank of rational quadratic form that represents all positive definite quadratic forms of ranks $\le$ $d$ (see, for example, \cite{Oh}). Therefore, there exists a universal $\Z$-form over $K$ of rank at most $r$. By \cite[Corollary 9]{KY3}, the result follows, i.e., there are only finitely many totally real number fields of degree $d$ with a universal quadratic form of rank $r$ over $K$.
\end{proof}

We now prove Theorem $\ref{cor:z-forms}$.

\begin{proof}[Proof of Theorem $\ref{cor:z-forms}$]

	If $Q$ is a $\Z$-form that is universal over $K$, where $[K:\Q]\le 5,$ then, as in the argument before, each totally positive integer in $K$ is represented by a $\Z$-form $Q'$ of rank $\le 5$. The Gram matrix of $2Q'$ is an integer matrix, and according to Mordell and Ko \cite {Ko}, 
	it is represented by the sum of squares of integer linear forms (note that this is no longer true for quadratic forms of rank 6 and higher). Therefore, all totally positive integers in $K$ that are divisible by $2$ are represented by the sum of squares. This restricts $K$ to the number fields listed in Theorem \ref{Thm:o2}. It remains to show that the fields not listed in Theorem \ref{cor:z-forms} do not admit a universal $\Z$-form.
	
	For that, we will use the fact that $2\alpha \in \co_K^+$ being represented by the sum of squares does not necessarily imply that $\alpha$ is represented by a $\Z$-form. Concretely, assume that there indeed is a universal $\Z$-form $Q$ of some rank $r$ over $K$ (that, in particular, represents $\alpha$),
	let $M_{Q}$ be its Gram matrix, 
	and let $\omega_1,\ldots, \omega_d$ be the integral basis of $\co_K$. We have $2M_{Q}\in \Z^{r\times r}$ and $(2M_{Q})_{ii}\in 2\Z$ for every $i$.
	
	Now, all solutions $v$ to the equation $v^tM_{Q}v=\alpha$ can be written as $Vw$ for $w^t=(\omega_1,\ldots,\omega_d)$ and some $V\in \Z^{r\times d}$. Denote $M_{Q'}=V^tM_QV$ so that $2\alpha=2w^tM_{Q'}w$, and observe that $2M_{Q'}\in \Z^{d\times d}$ and $(2M_{Q'})_{ii}\in 2\Z$ for every $i$.
	From the above-mentioned result of Mordell and Ko, there exists a rectangular integer matrix $U\in \Z^{s\times d}$ such that $2M_{Q'}=U^tU$; we can assume that $U$ has no rows containing all 0 (by omitting such rows if necessary). It follows that 
	\[2\alpha=w^tU^tUw \ \text{ and } \   (U^tU)_{ii}\in 2\Z \ \text { for all } \ i=1,\ldots,d.\] 
	This condition can be checked computationally: Observe that if we denote $w^tU^t=(x_1,\dots,x_s)$, then $2\alpha=x_1^2+\ldots+x_s^2$ and all $x_i\in\co_K$ and $x_i\neq 0$ (as $U$ did not contain empty rows).
	For a given element $\alpha$, there are only finitely many such decompositions (e.g., because $\Tr_{K/\Q}(x_i^2)\leq \Tr_{K/\Q}(2\alpha))$, and it is not hard to find all of them and check if the condition $(U^tU)_{ii}\in 2\Z$ is satisfied.

	For the numbers fields $\Q(\sqrt{2}),\ \Q(\sqrt{3}),\ \Q(\rho),\ \Q(\sqrt{2},\sqrt{5})$ and $\Q\left(\sqrt{(5+\sqrt{5})/2}\right)$, where $\rho$ is a root of $x^3-4x-2$, we found integers that do not have an appropriate solution. Specifically, such elements $\alpha\in\co_K^+$ are $2-\sqrt{2},\ 2+\sqrt{3},\ -2-3\rho+2\rho^2,\ 2-\sqrt{2}$ and $-2+\tau+2\tau^2-\tau^{3}$ (here $\tau=\sqrt{(5+\sqrt{5})/2}=\zeta_{20}+\zeta_{20}^{-1}$), respectively. 
\end{proof}

The second corollary of Theorem \ref{Thm:o2} relates to number fields that satisfy the local-global principle for representations of integers as sums of integral squares. 
The only condition for an integer $\alpha\in\co_K^+$ to be the sum of squares everywhere locally is that $\alpha$ is congruent to a square modulo $2$ (as quickly follows from \cite[Theorem~7.4]{Ri}; for an explicit reference, see \cite[Lemma~2.2]{KY2}).
The integer $\alpha\in\co_K^+$ is called 
an \textit{exceptional element} (or \textit{Ausnahme--Elemente} as defined by Scharlau in \cite{Sc1}) if it is congruent to a square modulo $2$, but it is not a sum of squared integers. Scharlau \cite{Sc1,Sc2} classified all totally real number fields without exceptional elements (i.e., satisfying the integral local-global principle for representations as sums of squares) of degrees $\leq 4$. 
In a number field without exceptional elements, all even integers are represented by the sum of squares, and so our Theorem \ref{Thm:o2} implies that there are no such quintic fields. This result was previously proved by Gerkmann \cite[Satz~2.3.2]{Ge}, but we nevertheless mention it here for completeness (as the thesis \cite{Ge} is not easily available).
\begin{corollary}[{\cite[Satz~2.3.2]{Ge}}]\label{th:exc}
	Let $K$ be a totally real number field of degree $d\le 5$. Then $K$ does not have any exceptional elements if and only if $K$ is isomorphic to 
	\begin{equation*}
		\Q,\ \Q(\sqrt{2}),\ \Q(\sqrt{3}),\ \Q(\sqrt{5}),\ \Q(\sqrt{2},\sqrt{5}),\text{ or }\ \Q\left(\sqrt{(5+\sqrt{5})/2}\right).
	\end{equation*}
\end{corollary}
	\section{The quartic case}\label{s:4}
	Let us now consider the number fields of degree four. We will use the Minkowski-reduced basis to construct totally positive integers of small trace and degree four over the rationals. For $\alpha \in K$, let us define $\|\alpha\|=|\sigma_1(\alpha)|+\cdots+|\sigma_d(\alpha)|$, where $|\cdot|$ on the right is the ordinary absolute value. 
	Note that $\|\alpha\|\asymp \house{\alpha}$ (with the implied constant depending on $d$).

	\begin{prop}\label{pr:quartic traces} Fix a squarefree integer $D>1$ and 
		assume that $K\ni\sqrt D$ is a totally real quartic field with discriminant $\Delta$.
		
		a) There is $\alpha\in\co_K^+\setminus\Q(\sqrt D)$ with
		$$\Tr_{K/\Q}(\alpha)\asymp \Delta^{1/{4}}\ (\text{as }\Delta\rightarrow\infty).$$
		
		b) For every $\beta\in\co_K\setminus\Q(\sqrt D)$, we have 
		$$\Tr_{K/\Q}(\beta^2)\gg\Delta^{1/2}\ (\text{as }\Delta\rightarrow\infty).$$
	\end{prop}
	Here, all the implied constants depend only on $D$ but not on $F$.
	\begin{proof}
		Order the real embeddings of $F$ so that 
		$\sigma_i(\sqrt D)=\sqrt D$ for $i=1,2$ and  $\sigma_i(\sqrt D)=-\sqrt D$ for $i=3,4$.

		We will use the properties of \textit{Minkowski reduced basis} \cite[Lecture X, \S 6]{Si5},
		\cite[Section 3]{BS+}. Specifically, in our situation, we will show that:
		
		\begin{claim}\renewcommand{\qedsymbol}{}
			
			There is a basis $\alpha_0=1,\alpha_1,\alpha_2,\alpha_3$ of $\co_K$ such that
			\begin{enumerate}[label={\rm (\arabic*)}]
				\item $\|1\|\leq \|\alpha_1\|\leq \|\alpha_2\|\leq \|\alpha_3\|$,
				\item $\|\alpha_i\|\ll \|\alpha\|$ for all $i=0,1,2$ and $\alpha\not\in\Z\cdot 1+\Z\cdot\alpha_1+\dots+\Z\cdot\alpha_{i-1}$,
				\item $\|\alpha_1\|\cdot \|\alpha_2\|\cdot \|\alpha_3\|\asymp\Delta^{1/2}$,
				\item $\|\alpha_2\|\asymp \|\alpha_3\|\asymp \Delta^{1/4}$,
				\item $\Z\cdot 1+\Z\cdot\alpha_1=\co_{\Q(\sqrt D)}$.
			\end{enumerate}
		\end{claim}
		
		\begin{proof}
			(1), (2), and (3) are basic properties of the Minkowski-reduced basis. 
			
			To establish (4), first note that \cite[Theorem~3.1]{BS+} gives $\|\alpha_3\|\ll \Delta^{1/4}$ (although \cite{BS+} work with a different length function, it is equivalent to ours).  
			We have $\sqrt D\in \co_K\setminus \Z\cdot 1$, and so by (2) we have $\|\alpha_1\|\ll \|\sqrt D\|\asymp 1$ (here the implied constant of course depends on $D$).  
			Together with $\|\alpha_2\|\leq \|\alpha_3\|\ll \Delta^{1/4}$ and (3), this implies (4).
			
			Finally, to prove (5), let $\co_{\Q(\sqrt D)}=\Z[\omega]$ with $\omega=\sqrt D$ or $(1+\sqrt D)/2$ depending on $D\pmod 4$. By (2) and (4) we have $\omega\in \Z\cdot 1+\Z\cdot\alpha_1$, and so $\co_{\Q(\sqrt D)}=\Z[\omega]\subset \Z\cdot 1+\Z\cdot\alpha_1$, but also $\alpha_1\in \Q\cdot 1+\Q\cdot\omega=\Q(\sqrt D)$. Thus $\alpha_1\in\co_K\cap\Q(\sqrt D)=\co_{\Q(\sqrt D)}$, finishing the proof of (5).
		\end{proof}
		
		Let us further work with such a {Minkowski-reduced basis} $\alpha_0=1,\alpha_1,\alpha_2,\alpha_3$.
		
		Let $k\in\Z$ be the smallest integer such that $k>\sigma_i(\alpha_2)$ for all $i$.
		Then we have $k\asymp\house{\alpha_2}$, $k+\alpha_2\in\co_K^+$, and $\Tr_{K/\Q}(k+\alpha_2)=4k+\Tr_{K/\Q}(\alpha_2)\asymp\house{\alpha_2}\asymp\|\alpha_2\|\asymp \Delta^{1/4}$, establishing a) for $\alpha=k+\alpha_2$.
		
		\medskip
		
		For b), let $\beta\in\co_K\setminus\Q(\sqrt D)$. As $\beta\not\in\Z\cdot 1+\Z\cdot\alpha_1=\co_{\Q(\sqrt D)}$, by (2) we have $\Delta^{1/4}\asymp\|\alpha_2\|\ll\|\beta\|$. Therefore, there is $i$ such that $|\sigma_i(\beta)|\gg\Delta^{1/4}$, and so $\Tr_{K/\Q}(\beta^2)\gg\Delta^{1/2}$. 
	\end{proof}  

	\begin{corollary}\label{cor:quart} 
		Fix a squarefree integer $D>1$ and 
		assume that $K\ni\sqrt D$ is a totally real quartic field with discriminant $\Delta$. 
		Let $\ve>0$.
		For each $m\in\Z$, $0<m\ll\Delta^{1/4-\ve}$ (where the implied constant depends on $D,\ve$), there is an element of $m\co_K^+$ that is not the sum of squares of elements of $\co_K$.
	\end{corollary}
	
	\begin{proof}
		Consider the element $\alpha$ from Proposition \ref{pr:quartic traces}a) and assume that $m\alpha=\sum\gamma_j^2$ is the sum of squares.
		We have $\Tr_{K/\Q} (m\alpha)\ll \Delta^{1/2-\ve}$, and so by Proposition \ref{pr:quartic traces}b) each $\gamma_j\in\Q(\sqrt D)$. But then $\alpha\in\Q(\sqrt D)$, a contradiction.	
	\end{proof}
	
\section{$2^n$-th roots of unity}\label{s:6} Throughout this section, we take 
$q=2^n$ with $n\ge 3$. Recall that we already introduced the notation for the real cyclotomic field $K_{2^n}=\Q(\omega)$ and
for its ring of integers $\co_{2^n}=\Z[\omega]$, where $\omega=\zeta_{{2^n}}+\zeta_{{2^n}}^{-1}$. We further denote $\omega_0=1$ and $\omega_j=\zeta_{q}^j+\zeta_{q}^{-j}$ for $j\neq 0$ (see Section \ref{s:2}). 

We will study the following two integers:
\[\alpha_n=\omega+2\; 
\mbox{ and }\;\gamma_n=\omega+\omega_{d-1}+3,\]
where $\gamma_n$ is defined for $n>3$ (as seen in \cite{Sc2}) and $d=2^{n-2}$. Note that we have \[\omega_d=\zeta^{2^{n-2}}_{2^n}+\zeta^{-2^{n-2}}_{2^n}=i-i=0,\] since $(\zeta^{2^{n-2}}_{2^n})^2=(\zeta^{-2^{n-2}}_{2^n})^2=-1$. We will eventually demonstrate that $\alpha_n\gamma_n$ is an indecomposable integer, for which we start by recalling the following result.

\begin{prop}\cite[p.~744]{We}\label{prop:2}
	All totally positive units in $\co_{2^n}$ are squares. Furthermore, 
 for all $\alpha,\alpha' \in \co_{2^n}$, if $\Nr_{K_{2^n}/\Q}(\alpha)=\Nr_{K_{2^n}/\Q}(\alpha')=2$, then $\alpha= \varepsilon \alpha'$ for some unit $ \varepsilon\in \co_{2^n}$. 
\end{prop}

Next we need several results on the properties fo the elements $\alpha_n$ and $\gamma_n$. Some of them are due to Scharlau \cite[part 5A]{Sc2}; we include also their proofs here for completeness (as the thesis \cite{Sc2} is not easily available).

\begin{prop}\label{prop:alpha2n}
	We have:
	\begin{enumerate}[label=(\roman*)]
		\item $\alpha_n\succ 0,$
		\item $\Tr_{K_{2^n}/\Q}(\alpha_n)=2d,$ 
		\item $\Nr_{K_{2^n}/\Q}(\alpha_n)=2$.		
	\end{enumerate}
\end{prop}
\begin{proof}
	$(i)$ follows from the fact that $\house{\omega}<2$. 
	
	$(ii)$ The minimal polynomial of $\zeta_{2^n}$ is $x^{2^{n-1}}+1$. Hence, it is easy to conclude that for $j=1,\dots, d-1$ we have $\Tr_{\Q(\zeta_{2^n})/\Q}(\zeta_{2^n}^j)=0$ (say, using Newton's identities). Given that $\omega=\zeta_{2^n}+\zeta_{2^n}^{-1}\in K_{2^n}$, we have \[0=\Tr_{\Q(\zeta_{2^n})/\Q}(\omega_j)=[\Q(\zeta_{2^n}):K_{2^n}]\Tr_{K_{2^n}/\Q}(\omega_j)=2\Tr_{K_{2^n}/\Q}(\omega_j).\]
	This immediately implies the value of the trace of $\alpha_n$. 
	
	$(iii)$
	By substituting $-1$ in the minimal polynomial of $\zeta_{2^n}$, we get that the norm of $1+\zeta_{2^n}$ in $\Q(\zeta_{2^n})$ is $2$. As the relative norm $\Nr_{\Q(\zeta_{2^n})/K_{2^n}}(1+ \zeta_{2^n})=2+ \omega$, $(iii)$ follows. 
\end{proof}

\begin{lemma}\label{lemma:sqr} Let $\beta=\sum_{i=0}^{d-1} a_i\omega_i, a_i\in\Z$. Then:
	\begin{enumerate}[label=(\roman*)]
		\item 
		$\Tr_{K_q/\Q}(\beta^2)=
		d(a_0^2+2\sum_{i= 1}^{d-1}a_i^2).$
		\item 	
		If $\beta=\sum_{j=1}^k \beta^2_j$ for some $\beta_1,\dots,\beta_k\in \co_{2^n}$, then for all odd $i$ we have $a_i\equiv 0 \pmod{2}$.
	\end{enumerate}
\end{lemma}
\begin{proof}
	$(i)$ Note that for $0<j<l\le d-1$, we have
	\begin{align*}
		\omega_j^2&=2+\omega_{2j}\\
		\omega_j\omega_l&=\omega_{l+j}+\omega_{l-j}.
	\end{align*}
	Therefore, 
	\begin{align}\label{eq:4}
		\left(\sum_{i=0}^{d-1} a_i\omega_i\right)^2&=a_0^2+2a_0\sum^{d-1}_{i=1}a_i\omega_i+\sum_{i=1}^{d-1}a_i^2(2 +\omega_{2i})+2\sum_{1\leq j<i\leq d-1}  a_ia_j(\omega_{i+j}+\omega_{i-j}).
	\end{align}

	Therefore, the formula for the trace follows from the fact that $\Tr_{K_{2^n}/\Q}(\omega_i)=0$ for $i\neq 0$.
	
	$(ii)$ using \eqref{eq:4} again, we have
	\begin{align*}
		\beta_l^2=\left(\sum_{i=0}^{d-1} b_i\omega_i\right)^2=b_0^2+\sum_{j=1}^{d-1}b_j^2(2 +\omega_{2j})+2\sum_{0\leq j<i\leq d-1} b_ib_j(\omega_{i+j}+\omega_{i-j})
	\end{align*}
	for any $b_i\in \Z.$ 
	
	As $\beta$ is the sum of several elements of this form, we see that  $a_i\equiv 0\pmod{2}$ for odd $i$.
\end{proof}

Further, we now prove the following.

\begin{prop}\label{cor:2squares}
	For $n>3$, $2\alpha_{n}$ cannot be represented as the sum of squares in $\co_{2^n}.$
\end{prop} 

\begin{proof}
	Let $2\alpha_n=\sum\beta_i^2,$ with $\beta_i\in \co_{2^n}.$ Given that the trace of $2\alpha_n$ is $4d$, by Lemma \ref{lemma:sqr}$(i)$, the only possibilities for $\beta_i$ are $\pm 1, \pm 2, \pm \omega_j, \pm 1\pm \omega_j$, or $\pm \omega _j \pm \omega_{j'}$. The respective traces are $\Tr_{K_{2^n}/\Q}(\beta_i^2)=d,4d,2d,3d$, or $4d$. Let us now go through all the possible decompositions.
    
    Clearly, $2\alpha_n\not = 4$. This also excludes the possibility of the sum of four squares of $\pm 1$. 
    
    Another case is that $2\alpha_n=(\pm \omega _j \pm \omega_{j'})^2$, but then
	\begin{align*}
	2\omega + 4&=2\alpha_n\\
	&=(\pm \omega _j \pm \omega_{j'})^2\\
	&=\omega _j^2+\omega_{j'}^2\pm 2\omega_{j}\omega_{j'}\\
	&=4+\omega_{2j}+\omega_{2j'}\pm 2 (\omega_{j+j'}+\omega_{j-j'})
	\end{align*}
	which is impossible since the representation in an integral basis is unique. 
    
    This leaves the cases $1^2+(\pm 1\pm \omega_j)^2, \omega_j^2+\omega_{j'}^2$ and $1^2+1^2+\omega_j^2$, which give us $4\pm 2\omega_j+\omega_{2j}$, $4+\omega_{2j}+\omega_{2j'}$, and $4+\omega_{2j},$ respectively. And only the first sum can equal to $2\alpha_n$, in the excluded case when $q=2^3$ and $j=1$, i.e., $4+2\omega_1+{\omega_4}$ (as $\omega_4=0$).
\end{proof}

 To proceed, it will be useful to complement Proposition \ref{prop:alpha2n} with similar information about the element $\gamma_n$ (cf. \cite[Lemma~4 in 5A]{Sc2}).
 
\begin{prop}\label{prop:gamma} We have:
\begin{enumerate}[label=(\roman*)]
    \item $\gamma_n\succ 0$,
    \item  $\Tr_{K_{2^n}/\Q}(\gamma_n)=3d$, 
    \item $\Nr_{K_{2^n}/\Q}(\gamma_n)=2^d+1.$
\end{enumerate}
\end{prop}

\begin{proof} The proofs of $(ii)$ and $(iii)$ will be used to derive $(i)$. Specifically, we shall show that $\Tr_{K_{2^n}/K_{2^{n-1}}}(\gamma_n),\Nr_{K_{2^n}/K_{2^{n-1}}}(\gamma_n)\succ 0$, and so each defining polynomials in our tower of quadratic extensions has totally positive roots. This is enough to conclude that $\gamma_n\succ 0$.

For $(ii)$, note that $\Tr_{K_{2^n}/K_{2^{n-1}}}(\omega_j)=\Tr_{K_{2^n}/\Q}(\omega_j)=0$ for all $j$. And directly from the definition of $\gamma_n$ it follows that 
\begin{align*}
    \Tr_{K_{2^n}/\Q}(\gamma_n)&=\Tr_{K_{2^n}/\Q}(\omega+\omega_{d-1}+3)\\&=3d.
\end{align*} 
Thus, follows $(ii)$, and shows that the relative trace of $\gamma_n$ is totally positive, which we require for $(i)$. 

For $(iii),$ we begin by observing that
	\begin{align*}
	(\omega+\omega_{d-1})^2&=\omega^2+2\omega\omega_{d-1}+\omega_{d-1}^2\\
	&=4+\omega_2+\omega_{2d-2}+2\omega_d+2\omega_{d-2}\\
	&=4+2\omega_{d-2}.
	\end{align*} 
	We used that $\omega_d=0$ and $\omega_2=\zeta^2_{2^n}+\zeta^{-2}_{2^{n}}=-(\zeta^{2^{n-1}+2}_{2^n}+\zeta^{-2^{n-1}-2}_{2^n})=-\omega_{2d-2}$, as $\zeta^{2^{n-1}}_{2^n}=-1$.
    
Therefore, 
\begin{align*}
\Nr_{K_{2^n}/K_{2^{n-1}}}(\gamma_n)&=9-(\omega+\omega_{d-1})^2\\
&=5-2\omega_{d-2}.
\end{align*}
As $\house{\omega_i}<2$, the above integer is totally positive. Therefore, the relative norm of $\gamma_n$ is totally positive and $(i)$ follows, i.e., $\gamma\succ 0.$ 
To complete the proof of $(iii)$ we compute another relative norm,  
\begin{align*}
\Nr_{K_{2^{n-1}}/K_{2^{n-2}}}(5-2\omega_{d-2})&=25-4\omega_{d-2}^2.
\end{align*}
If $n=4$, then $d=4$ and $\omega_2=\sqrt{2}$, and thus the above is $17=2^4+1$. Furthermore, using $\omega_{2d-2}^2=2+\omega_{2d-4}$, we can write the norm above as:
		\begin{align*}
	\Nr_{K_{2^n}/K_{2^{n-2}}}(\gamma_n)&=17-4\omega_{2d-4}\\
	&=2^4+1-2^2\omega_{2d-4}\\
	&=2^{2^2}+1-2^2\omega_{2d-2^2}.
	\end{align*}
	We claim that 
 \begin{equation*}
     \Nr_{K_{2^n}/K_{2^{n-j}}}(\gamma_n)=2^{2^j}+1-2^{2^{j-1}}\omega_{2^{j-1}d-2^j},
 \end{equation*} which we prove by completing the successive norm:
		\begin{align*}
		(2^{2^j}+1)^2-(2^{2^{j-1}}\omega_{2^{j-1}d-2^j})^2&=2^{2^{j+1}}+2^{2^j+1}+1-2^{2^j+1}-2^{2^j}\omega_{2^jd-2^{j+1}}\\
		\Nr_{K_{2^n}/K_{2^{n-j-1}}}(\gamma_n)&=2^{2^{j+1}}+1-2^{2^j}\omega_{2^jd-2^{j+1}}.
	\end{align*}
	Now, when $j=n-2$, then $K_{2^{n-j}}=K_{4}=\Q$. Furthermore, for this $j$, we have $\omega_{2^{j-1}d-2^j}=0$, and the result follows.
\end{proof}

Now we begin proving that $\alpha_n, \gamma_n$ and $\alpha_n\gamma_n$ are indecomposable.

\begin{prop}\label{cor:alpha} Let $F$ be a totally real number field that contains $K_{2^n}$. Then 
	$\alpha_n$ is indecomposable in $F$. 
\end{prop}
\begin{proof}
According to $(ii)$ in Proposition~\ref{prop:alpha2n}, we have that $\Tr_{F/\Q}(\alpha_n)=2d[F:K_{2^n}]=2[F:\Q]$, and so $\alpha_n$ is indecomposable by part c) of Theorem~\ref{cor:5/2}.
\end{proof}

\begin{prop}\label{cor:gamma} Let $F$ be a totally real number field that contains $K_{2^n}$. Then 
    $\gamma_n$ is indecomposable in $F$.
\end{prop}
\begin{proof}
By part $(iii)$ of Proposition~\ref{prop:gamma} we have $\Nr_{K_{2^n}/\Q}(\gamma_n)=2^d+1$, so by Proposition~\ref{lemma:norm} (used with $c=3/2$), $\gamma_n$ is indecomposable in $\co_{2^n}$. Let $L$ be a totally real number field such that $\gamma_n \in L$ and $[L:\Q]=d'$. Then, from Proposition~\ref{prop:gamma}$(ii)$, $\Tr_{L/\Q}(\gamma_m)=3d[L:K_{2^n}]=3d'$. In particular, if $\gamma_n$ is indecomposable, then from part a) of Theorem \ref{thm:1} we have that either $\gamma_n=\beta+1$ or $\gamma_n=2\phi+2$, where $\beta \in L$. However, that would imply that $\beta,\phi\in K_{2^n}$, a contradiction with the indecomposability of $\gamma_n$ in $K_{2^n}$. 
\end{proof}

\begin{lemma}\label{lemma:3}
    $3$ is inert in $K_{2^n}$. In particular, there does not exist $\beta\in\co_{2^n}$ with $\Nr_{K_{2^n}/\Q}(\beta)= 3$.
\end{lemma}
\begin{proof} 
	By \cite[Chapter~4, Exercise~12]{Marcus}, we have that the inertial degree of 3 in $K_{2^n}$ equals the smallest positive integer $f$ satisfying $3^f\equiv \pm 1\pmod{2^n}$ (for, in the notation of the exercise, we have $m=2^n$ and the corresponding subgroup $H=\{\pm 1\}$). To show that 3 is inert, we therefore want to show that $f=2^{n-2}$.
	
	First, note that $3^j\not\equiv -1\pmod{2^n}$ for every even $j$, as already $3^j\equiv (-1)^j=1\not\equiv -1\pmod{4}$. Thus, if $3^f\equiv - 1\pmod{2^n}$, then $f$ is odd. But then $2f$ is the order of $3$ in $\Z_{2^n}^\times$, and so $2f$ is a power of 2. This can happen for odd $f$ only when $f=1$, which is, however, impossible. 
	
	Thus, $f$ is the smallest positive integer such that $3^f\equiv 1\pmod{2^n}$. We now prove by induction on $n$ that $f=2^{n-2}$. 
	We directly verify this for $n=3,4.$  
	
	Now, assume that it holds for $n\ge 4$ and assume that $3^j\equiv 1 \pmod{2^n}$ for some $j$. Then 
	$$
		(3^{j/2}-1)(3^{j/2}+1)\equiv 0\pmod{2^n}.
	$$ 
	First, note that $j$ is a power of two, and from case $n=4$, we have $j\ge 4$. 
	Therefore, $2\mid (3^{j/2}+1)$, but $4\nmid(3^{j/2}+1) $ as we already saw. By the induction hypothesis, if $j=2^{n-2},$ then $2^{n-1}\mid (3^{j/2}-1),$ and it follows that $f\le 2^{n-2}.$ But if $j<2^{n-2}$, we have $2^{n-1}\nmid (3^{j/2}-1)$ again by the induction hypothesis. The result follows as $4\nmid (3^{j/2}+1).$ 

Concerning the ``In particular'' part, if there was such $\beta$, then $\beta\mid 3$. But $3\co_{2^n}$ is prime and $\Nr_{K_{2^n}/\Q}(3)= 3^{2^{n-2}}$ with $n\geq 3$, so this is impossible. 
\end{proof}

\begin{prop}\label{prop:alpha*gamma}
$\alpha_n\gamma_n$ is indecomposable in $K_{2^n}$.
\end{prop}
	\begin{proof}
	    Let us assume that $\alpha_n\gamma_n$ decomposes as $\beta_1+\cdots+\beta_k$ with $k\geq 2$. Recall that $\Nr_{K_{2^n}/\Q}(\alpha_n\gamma_n)$ $=2(2^d+1)$, and so looking at the norm bound from Lemma \ref{lemma:norm_bound}, we can conclude that $\Nr_{K_{2^n}/\Q}(\beta_i)\le 3$. As there is no element of norm 3 by Lemma \ref{lemma:3}, we have $\Nr_{K_{2^n}/\Q}(\beta_i)\le 2$. If one $\beta_i$ has a norm 2, then by Proposition \ref{prop:gamma}, we find that it is a (totally positive) unit multiple of $\alpha_n$. Hence
\[\alpha_n\gamma_n-\beta_i=\alpha_n(\gamma_n- \varepsilon^2)\succ 0,\]
where $ \varepsilon$ is a unit (using that all totally positive units are squares by the same proposition). But since $\alpha_n\succ 0$ and $\gamma_n$ are indecomposable, this is impossible. Therefore, $\Nr_{K_{2^n}/\Q}(\beta_i)=1$ for all $i$, and in particular, if $\alpha_n\gamma_n$ is decomposable, then it is represented by the sum of squares, as all units are squares. We have
	    \begin{align*}
	       \alpha_n\gamma_n&=(\omega+2)(\omega+\omega_{d-1}+3)\\
	       &=\omega^2+\omega\omega_{d-1}+5\omega+2\omega_{d-1}+6\\
	       &=5\omega_1+\omega_2+\omega_{d-2}+2\omega_{d-1}+8.
	    \end{align*}
	    However, from Lemma \ref{lemma:sqr}$(ii)$, $\alpha_n\gamma_n$ is not a sum of squares, as it contains the term $5\omega_1$ with an odd coefficient 5.
	    \end{proof}

	\begin{theorem}\label{thm:2-K}
    For $n>3$, there are no universal classical ternary quadratic forms over $\co_{2^n}$.
	\end{theorem}

	\begin{proof}
	    Let us assume for contradiction that $Q$ is a universal ternary quadratic form over $\co_{2^n}$. Lemma \ref{lemma:unit} gives us 
	    \begin{equation*}
	        Q\cong \langle 1 \rangle\perp Q', 
	    \end{equation*}
	where $Q'$ is a binary quadratic form. By Propositions \ref{cor:alpha}, \ref{cor:gamma}, and \ref{prop:alpha*gamma}, $\alpha_n,\gamma_n$, and $\alpha_n\gamma_n$ are indecomposable in $\co_{2^n}$.  Furthermore, the norms of these elements are nonsquare (Propositions \ref{prop:alpha2n}$(iii)$ and \ref{prop:gamma}$(iii)$), and so also $\alpha_n,\gamma_n$ and $\alpha_n\gamma_n$ are not squares in $\co_{2^n}$. This implies that $Q'$ has to represent them. 
	
	Thus, there exists $v\in \co_{2^n}^2$ such that $Q'(v)=\alpha_n$. In particular, since the norm of $\alpha_n$ is 2, $v$ is primitive; otherwise, if $v=\delta w$, we would have $\alpha_n=\delta^2Q(w)$. Applying Theorem in \cite{GMR} we can construct an invertible matrix $X$ with the first row corresponding to $v$, so that $XMX^t=\begin{pmatrix} Q'(v) & \beta\\ \beta  & \gamma \end{pmatrix},$ where $M$ is the Gram matrix of $Q',$ and $\beta,\gamma \in \co_{2^n}$. Therefore, without loss of generality, assume that 
	\[Q'\cong \begin{pmatrix} \alpha_n & \beta\\ \beta & \gamma \end{pmatrix}.\]
	First, we examine the decomposition of the discriminant of $Q'$. Let $\alpha_n\gamma-\beta^2=\sum^k_{i=1}m_i \kappa_i$, where $m_i\in \Z^+$ and $ \kappa_i$ are distinct indecomposable integers. We claim that each $ \kappa_i$ is either a square or $\alpha_n$ times a square. If there exists $i$, say $i=1$, such that $\kappa_1$ is not a square, then $Q'$ represents $ \kappa_1$, that is,
	\begin{equation}\label{eq:eps}
	\alpha_nx^2+2\beta xy+\gamma y^2= \kappa_1.
	\end{equation}
	By multiplying by $\alpha_n$ and completing the square, we get the following:
	\begin{align*}
	    (\alpha_n x+\beta y)^2+(\alpha_n\gamma -\beta^2)y^2=&\alpha_n \kappa_1\\
	    (\alpha_n x+\beta y)^2+y^2\sum^k_{i=1}m_i \kappa_i=&\alpha_n \kappa_1\\
	     (\alpha_n x+\beta y)^2+y^2\sum^k_{i=2}m_i \kappa_i=&(\alpha_n-m_1y^2) \kappa_1.
	\end{align*}
	Given that the left-hand side is $\succeq 0$, it implies that $\alpha_n-m_1y^2\succeq 0$. Hence $y=0$, as $\alpha_n$ is indecomposable and $N_{K^{2^n}/\Q}(\alpha_n)=2$, and equation (\ref{eq:eps}) gives us that $\alpha_nx^2= \kappa_1$, proving the claim. Therefore, $\alpha_n\gamma-\beta^2=\sum x_i^2+\alpha_n\sum y_i^2$, $x_i,y_i\in \co_{2^n}$.
	
	Completing the square and recalling that $Q'$ represents $\gamma_n$, we have that
	\begin{equation*}\label{eq:2}
	    (\alpha_n u+\beta v)^2+(\alpha_n\gamma -\beta^2)v^2=\alpha_n\gamma_n.
	\end{equation*}
	Recall that $\alpha_n\gamma_n$ is indecomposable and not square. Therefore,
	\[ (\alpha_n u+\beta v)^2=0,\]
	and
	\[v^2\sum^k_{i=1}m_i \kappa_i=\alpha_n\gamma_n,\]
	where we must have $k=1$ and $m_1=1$. 
	
	However, we already saw that $ \kappa_1=\lambda^2$ or $\alpha_n\lambda^2$ for some $\lambda$. Thus, $\alpha_n\gamma_n=v^2 \kappa_1=v^2\lambda^2$ is a square, or $v^2 \kappa_1=\alpha_n\lambda^2v^2=\alpha_n\gamma_n$, thus $\gamma_n$ is a square. In both cases, we get a contradiction.
	\end{proof}
	
\section{$p$-th roots of unity}\label{s:7} Throughout this section, let $p\geq 5$ be a prime number, and let $K_p=\Q(\zeta_{p}+\zeta_{p}^{-1})$ be the maximal totally real subfield of $p$-th cyclotomic field. We have that $[K_p:\Q]=d=(p-1)/2$, and we denote by $\co_p$ the ring of integers of $K_p$.
 
 Recall that $\omega=\zeta_{p}+\zeta_{p}^{-1}$, $\omega_j=\zeta_{p}^j+\zeta_{p}^{-j}$; then $\omega=\omega_1,\dots,\omega_{d}$ form an integral basis for $\co_p$.
 
 Let us start by establishing a variant of Proposition \ref{prop:alpha2n} and Lemma \ref{lemma:sqr}. These results are again due to Scharlau \cite[part 5B]{Sc2} and we include their proofs for completeness (as the thesis \cite{Sc2} is not easily available).
 
 \begin{prop}\label{prop:alphap}
 	Let $\alpha_p=
 		2 - \omega.$
 	
 	We have:
 	\begin{enumerate}[label=(\roman*)]
 		\item $\alpha_p\succ 0,$
 		\item $\Tr_{K_p/\Q}(\alpha_p)=
 			2d+1,$ 
 		\item $\Nr_{K_p/\Q}(\alpha_p)=
 			p.$
 		
 		\item For $\beta=\sum_{j=0}^{d-1} a_j\omega_j\in \co_p$ the following holds
 		\[\Tr_{K_p/\Q}(\beta^2)=
 			\sum_{j=0}^{d-1} a_j^2 +2\sum_{0\leq j<i\leq d-1} (a_i-a_j)^2	.\]
 	\end{enumerate}
 \end{prop}
 \begin{proof}
 	
 	$(i)$ follows from the fact that $\house{\omega}<2$. 
 	
 	$(ii)$ The minimal polynomial of $\zeta_p$ is $x^{p-1}+x^{p-2}+\dots+1$. Therefore, for $j=1,\dots, d-1$ we have $\Tr_{\Q(\zeta_p)/\Q}(\zeta_p^j)=
 		-1.$
 		
 	As $\omega=\zeta_p+\zeta_p^{-1}\in K_p$, we have \[-2=\Tr_{\Q(\zeta_p)/\Q}(\omega_j)=[\Q(\zeta_p):K_p]\Tr_{K_p/\Q}(\omega_j)=2\Tr_{K_p/\Q}(\omega_j).\]
 	The trace of $\omega$ is thus $-1$, which immediately implies the value of the trace of $\alpha_p$. 
 	
 	$(iii)$	By substituting $-1$ in the minimal polynomial of $\zeta_p$, we get that the norm of  $1-\zeta_p$ is $p$. As the relative norm $\Nr_{\Q(\zeta_p)/K_p}(1- \zeta_p)=2- \omega$,  part $(iii)$ follows.

 	$(iv)$ follows from  \cite[Proposition 4.1$(a)$]{BN} that $(\Tr_{K_p/\Q}(xy))\cong pI_d-2J_d$, where $I_d$ is an identity matrix of dimension $d$, and $J_d$ is a $d\times d$ matrix of all ones. Due to our choice of basis, they are equal (as opposed to just equivalent), and the proof follows trivially.
 \end{proof}
 
\begin{lemma}[{\cite[5B, Lemma 8]{Sc2}}]\label{lemma:ptrace}
	Assume that $p\ge 29$ and that $\beta\in\co_p$ is such that $\Tr_{K_p/\Q}(\beta^2)<4d+2$. Then the only possibilities for $\pm\beta$ and the trace are as follows (with $i,h=1,\dots,d,i\neq h$):

\begin{center}	
\begin{tabular}{|c|c|}
	\hline
	$\pm\beta$ & $\Tr_{K_p/\Q}(\beta^2)$ \\
	\hline
	$0$ & $0$ \\
	$1$ & $d$ \\
	$\omega_i$ & $2d-1$ \\
	$\omega_i+1$ & $3d-3$ \\
	$\omega_i-1$ & $3d+1$ \\
	$\omega_i+\omega_h$ & $4d-6$ \\
	$2$ & $4d$ \\
	\hline
\end{tabular}
\end{center}
\end{lemma}

Note that $1=-\sum \omega_j$, and so if we express the elements listed in the lemma in terms of the integral basis $\omega_1,\dots, \omega_d$, we get the elements 	$1=-\sum \omega_j,\
\omega_i+1=\omega_i-\sum \omega_j,\  
\omega_i-1=\omega_i+\sum \omega_j,\  
2=-2\sum \omega_j$ (where all the sums over $j$ run from $1$ to $d=(p-1)/2$).

\begin{proof}
Assume that $\beta=\sum a_i\omega_i$  satisfies $\Tr_{K_p/\Q}(\beta^2)<4d+2$, and let exactly $k$ of the coefficients $a_i$ be nonzero.

Proposition \ref{prop:alphap}$(iv)$ says that \[\Tr_{K_p/\Q}(\beta^2)=	\sum a_j^2 +2\sum_{i>j} (a_i-a_j)^2.\]
There are $k(d-k)$ ordered pairs $(a_i,a_j)$ such that exactly one of $a_i,a_j$ is nonzero, and each such pair contributes at least $1$ to $\sum_{i>j} (a_i-a_j)^2$. Thus 
\[4d+2>\Tr_{K_p/\Q}(\beta^2)\geq k + 2k(d-k)=-2k^2+k(2d+1).\]

As a function of $k$, the right-hand side increases for $k\leq d/2+1/4$ and decreases for $k\geq d/2+1/4$. Moreover (using our assumption that $p\ge 29$), the inequality does \textit{not} hold for $k=3$ and $k=d-2$.

\medskip 

This leaves only the cases $k=0,1,2,d-1,d$. The rest of the proof consists only of laborious verification. First, it is easily verified that if $k=0,1$, then $\pm\beta=0,\omega_i$.

Let $k=2$. Using the above formulas, we get $4d+2>\Tr_{K_p/\Q}(\beta^2)\geq \sum a_j^2 + 4(d-2)$, and so $\sum a_j^2<10$.
By the symmetry of the formula for $\Tr_{K_p/\Q}(\beta^2)$, we can assume that $a_1\geq 1, |a_1|\geq a_2\neq 0, a_3=\dots=a_d=0$.
Then the only possibilities for $(a_1,a_2)$ are $(2,\pm 2),(2,\pm1),(1,\pm1)$, among which only $(1,1)$ (corresponding to the elements $\omega_i+\omega_h$) gives a trace less than $4d+2$.

\medskip 

Let $k=d-1$. In this case, first, assume that the $k=d-1$ nonzero coefficients $a_i$ attain at least two distinct values. Then we can group these nonzero coefficients into two groups with $u=d-1-v\geq v\geq 1$ elements such that all the values in the first group are distinct from the values in the second group. There are then $uv$ ordered pairs of elements from different groups, and so $4d+2>\Tr_{K_p/\Q}(\beta^2)\geq d-1 + 2(d-1)+2uv$. Thus, $u+v+6=d+5> 2uv$, which is impossible: If $v=1$, then $u+7> 2u$, and if $v\geq 2$, then $2u+6\geq u+v+6>2uv\geq 4u$.

Thus, the only possibility when $k=d-1$ is that all the $k=d-1$ nonzero coefficients $a_i$ attain the same value. This value then must be $\pm 1$, leading to the element $\beta=\omega_i+1=\omega_i-\sum \omega_j$.

\medskip 
 
Finally, let $k=d$. If all $d$ nonzero coefficients attain the same value, then we just get $\pm\beta=1,2$ (note that $1=-\sum \omega_j$). Thus, assume that there are at least two distinct values and again group them into two groups with $u=d-v\geq v\geq 1$ elements to get $4d+2>d + 2uv$, i.e., 
$$3d+2>2u(d-u).$$ As $u\geq v$, we have $u\geq d/2$, which is the region where the right-hand side of the last displayed inequality is decreasing. The inequality does not hold for $u=d-2$, leaving us only with $u=d-1, v=1$.

But this holds for any grouping of the coefficients $a_i$ into two groups with distinct values, and so actually there can be only two distinct values, without loss of generality, say $0<a=a_1=a_2=\dots=a_{d-1}\neq a_d=b\neq 0$.

The formula for the trace then implies that $$4d+2>\left( (d-1)a^2+b^2\right) + 2(d-1)(b-1 )^2,$$ and so $a=1$. Then $b-1=\pm 1$, and so $b=2$, giving the element $\omega_i-1=\omega_i+\sum \omega_j$.
\end{proof}

Of course, one could easily continue with the arguments given above to also determine the squares with larger traces (in fact, below we will only need to know the result for traces $\leq 5d/2+2$).

\medskip

Consider now the element $\alpha_p=2-\omega\in\co_p$.

\begin{prop}\label{cor:psquares}
	 $2\alpha_p$ cannot be represented as the sum of squares in $\co_p$, where $p\geq 7$ is a prime. 
\end{prop}
\begin{proof}
	For primes less than $29$, we computationally check that it is indeed the case. Assume that $p\ge 29$. Let $2\alpha_p=\sum \beta^2_i$. The integer $2\alpha_p$ is not a square, as its norm is $2^dp$. Taking the traces, we have $4d+2=\sum \Tr_{K_p/\Q}(\beta_i^2)$. In particular, the traces of $\beta_i^2$ must be less than $4d+2$, and considering all the possible combinations from the above lemma, we conclude that $2\alpha_p$ cannot be represented by the sum of squares.
\end{proof}

\begin{lemma}\label{lemma:apu}
 For $p>7$, $2\alpha_p-1\not \succ 0$, and $2\alpha_p-\omega_j^2\not \succ 0$ for $j=1,\ldots, d-1$.
\end{lemma}
\begin{proof}
	First, we note that $2\alpha_p=4-2\omega$ and $\omega=2\cos (2\pi/p)$. So for $p>7$, using that $\cos(x)\approx 1-x^2$ for small $x$ and approaching $1$ as $p\rightarrow \infty$, while $\cos (2\pi/p)>\frac{3}{4}$, we get that \[2\alpha_p-1=4-4\cos (2\pi/p)-1<0. \]
	
	 For the latter inequality, we note that $\omega_j^2=2+\omega_{2j}$, to get the following:
	\begin{align*}
	2\alpha_p-\omega_j^2&=2-4\cos (2\pi/p)-2\cos(4 \pi j/p)\\&=2(1-2\cos (2\pi/p)-\cos(4 \pi j/p))
	\end{align*}
	for $j=1,\ldots,d-1$. 
	Therefore, it suffices to show that $1-2\cos (2\pi/p)-\cos(4 \pi j/p)\not\succ 0$. 
	
	 Let $m \equiv p-4j \pmod{ 2p}$ be an odd integer. Then we have a bijection between $j$'s and odd $m$'s. Using the congruence, it follows that
	\[\cos ( {4j \pi/p}  ) = \cos\pi  ( 1-m/p  )  = - \cos ( {\pi m/p}).\]
	
	Substituting it into the above and rearranging a bit, we find that $\cos   ( \pi m/p )\ge 2 \cos  ( 2 \pi/p )-1$.
	
	We have approximately $\cos(x)  \approx 1 - \frac{x^2}{2}$ (specifically, from Taylor expansion we get the inequality $1 - \frac{x^2}{2}+\frac{x^4}{24}\geq \cos (x) \geq 1 - \frac{x^2}{2}$). Thus, using the upper and lower bounds for the cosines, we get
	
\begin{align}
1- \frac{\pi^2 m^2}{2p^2}+\frac{\pi^4m^4}{24p^4}&\geq 1 - \frac{4 \pi^2}{p^2}\nonumber\\
\pi^2m^4-12p^2m^2+96p^2&\geq 0 \label{ineq:1}
\end{align}
	Notice that the above polynomial in $m$ is even, and so we can assume that $m>0$; by the congruence defining $m$, we can also take $m\leq p$. Moreover, for $p\ge 10$, the value of the polynomial alternates when evaluated at $m=2, 3$ and at $m=p, (p+2)$, i.e., $+,-$ and $-,+,$ respectively. Therefore, by Rolle's theorem, the only odd positive integer value of $m\le p$ for which the inequality (\ref{ineq:1}) holds is $m=1$. 
	
	We can write all conjugates of $1-2\cos (2\pi/p)-\cos(4 \pi j/p)$ as $1-2\cos (2\pi k /p)-\cos(4 \pi jk/p)$ for $k=1,\ldots,d$.	Now, returning back to $j$, we have $4j \equiv p-1 \pmod{2p}$. 
	
	Let us now look at the second conjugate, i.e., $1 \geq 2 \cos \left ( \frac{4 \pi}{p} \right ) + \cos \left ( \frac{8j \pi}{p} \right )$. Then \[\cos \left ( \frac{8j \pi}{p} \right ) = \cos \left ( \frac{2 \pi (p-1)}{p} \right ) = \cos \left ( \frac{2 \pi}{p} \right ).\] The above gives us $1 \geq 2 \cos \left ( \frac{4 \pi}{p} \right ) + \cos \left ( \frac{2 \pi}{p} \right )$, which is clearly false using inequality $\cos (x)\geq  1 - \frac{x^2}{2}$.
\end{proof}
Finally, the proof of the special case of Kitaoka's conjecture follows. We note that for primes $p\equiv 3 \pmod{4}$ the result was already known by Theorem 1 in \cite{EK}, as the corresponding degree of the number field is odd. However, the result for $p\equiv 1 \pmod{4}$ is new.
\begin{theorem}\label{thm:p-K}
	There are no universal classical ternary quadratic forms over $\co_p$, where $p>5$ is a prime number.
\end{theorem}

Note that the primes $2$ and $3$ can be excluded as they correspond to the trivial extension $\Q$. For $p=5$, the sum of three squares is universal over $\co_5=\Z[\frac{1+\sqrt 5}2]$ \cite{Ma}.

\begin{proof}
For $p=7$, $K_7$ is an odd-degree extension; ergo, there is no ternary universal quadratic form \cite[Lemma 3]{EK}. Thus, assume that $p>7$ and let $Q$ be a ternary universal quadratic form over $\co_p$. By Lemma \ref{lemma:unit}, we have
	\begin{equation}\label{eq:3}
	Q\cong \langle 1 \rangle\perp \underset{=Q'}{\begin{pmatrix} \alpha & \beta\\ \beta & \gamma \end{pmatrix}}.
	\end{equation}
If there exists a nonsquare totally positive unit $ \varepsilon \in\co_p$, we reapply the lemma to give us \[Q \cong \langle 1 \rangle\perp \langle \varepsilon \rangle\perp \langle\beta\rangle,\] where $\beta\in \co^+_p$. Without loss of generality, assume that $\beta=\alpha_p,$ as $\alpha_p$ is not a square and indecomposable. Noting that $ \varepsilon\alpha_p$ is also nonsquare indecomposable and not represented by $Q$, we have a contradiction.
	
	Hence, we assume that all totally positive units are squares, and thus we have units of every signature in $\co_p$ \cite[p. 111, Corollary 3]{Na}. Given that $p$ is an odd prime number, $2$ does not divide the discriminant of $\co_p$, and so does not ramify in $\co_p$ (that is, $2$ is squarefree in $\co_p$). The only decomposition of $2$ into indecomposable integers 
is $1+1$. Therefore, if $Q$ represents $2$, then $Q'$ represents $1$ or $2$:
 
If it represents $1$, by Lemma \ref{lemma:unit} follows that \[Q\cong \langle 1 \rangle\perp \langle 1 \rangle \perp \langle\beta \rangle,\] for some $\beta \in \co_p^+$. As the form has to represent $\alpha_p$, and $\alpha_p$ is a nonsquare indecomposable integer (as its trace is $2d+1$ and its norm is $p$ by Proposition \ref{prop:alphap}), it implies that $\langle\beta \rangle$ represents $\alpha_p$. Specifically, $\beta x^2=\alpha_p$, where $x$ must be a unit. Therefore, we can assume that $\beta=\alpha_p$. By the above corollary, $2\alpha_p$ is not a sum of squares. If $x^2+y^2+\alpha_pz^2=2\alpha_p$, with $z\not=0$. Then, $(2-z^2)\alpha_p=x^2+y^2\succeq 0$. This gives us a contradiction since the only decomposition of $2$ is $1+1$, and if $z^2=1$, then $\alpha_p$ is represented as a sum of squares.
	
	In the case where $Q'$ represents $2$ and $2$ is squarefree in $\co_p$, we can assume the following: \[Q'\cong 2x^2+2\beta xy + \gamma y^2,\] for $\beta,\gamma \in \co_p$. The form has to represent $\alpha_p$, hence 
 \[2\alpha_p=4x^2+4\beta xy + 2\gamma y^2=(2x+\beta y)^2+(2\gamma-\beta^2)y^2.\]
	
	We note that the trace of $2\alpha_p$ is $4d+2$. From part a) of Theorem \ref{thm:1}, we find that if \[ \Tr_{K_p/\Q}((2x+\beta y)^2)>\frac{5}{2}d+2,\] then $(2\gamma-\beta^2)y^2=0$ or $1$. Both cases are impossible, as $2\alpha_p$ would be represented as a sum of squares. We conclude that $\Tr_{K_p/\Q}((2x+\beta y)^2)\le \frac{5}{2}d+2$. For $p\ge 29$, by Lemma \ref{lemma:ptrace}, we have $(2x+\beta y)^2=0,1$ or $\omega_j^2$, as $\frac{5}{2}d+2< 3d-3$ for $d>10$ (or $p\ge 23$). It remains to deal with $p\equiv 1 \pmod{4}$ such that $7<p<29,$ i.e., $p=13, 17.$ For these primes, we computed that $\Tr_{K_p/\Q}((2x+\beta y)^2)\le \frac{5}{2}d+2$ is also only satisfied for $(2x+\beta y)=0, 1$ or $\omega_j^2$. By Lemma \ref{lemma:apu}, we know that $(2x+\beta y)=1$ or $\omega_j^2$ is impossible. 
	This implies that $2x=-\beta y$, and hence $\alpha_p=\gamma y^2+2x^2$. As $\alpha_p$ is an indecomposable integer, we have that $xy=0$ and so $\beta=x=0.$ Without loss of generality, assume that $Q\cong \langle 1 \rangle\perp \langle 2\rangle \perp \langle \alpha_p\rangle.$ By the same argument as before, we can show that $2\alpha_p$ is not represented by $\langle 1 \rangle\perp \langle 2\rangle \perp \langle \alpha_p\rangle$ using that $(2-z^2)\alpha_p=x^2+2y^2\succeq 0$, and therefore $Q$ cannot be universal.
\end{proof}

 \end{document}